\documentclass[reqno]{amsart}
\usepackage{amssymb}
\usepackage[mathscr]{eucal}
\usepackage{hyperref}

\theoremstyle{plain}
\newtheorem{prop}{Proposition}[section]
\newtheorem{lem}[prop]{Lemma}

\theoremstyle{definition}
\newtheorem{lab}[prop]{}

\newcommand{\C}{{\mathbb{C}}}
\newcommand{\Q}{{\mathbb{Q}}}
\newcommand{\R}{{\mathbb{R}}}
\newcommand{\Z}{{\mathbb{Z}}}

\DeclareMathOperator{\ch}{char}

\newcommand{\disc}{\mathrm{disc}}
\newcommand{\ol}{\overline}
\newcommand{\Label}[1]{\label{#1}}

\begin{document}

\title
[Hilbert's theorem on ternary quartics: Some complements]
{An elementary proof of Hilbert's theorem\\
on ternary quartics: Some complements}

\author{Albrecht Pfister}
\address
  {Institut f\"ur Mathematik,
  Johannes Gutenberg Universit\"at,
  Staudingerweg 9,
  55099 Mainz,
  Germany}
\email
  {pfister@mathematik.uni-mainz.de}
\author{Claus Scheiderer}
\address
  {FB Mathematik und Statistik,
  Universit\"at Konstanz,
  78457 Konstanz,
  Germany}
\email
  {claus.scheiderer@uni-konstanz.de}

\begin{abstract}
In 1888, Hilbert proved that every nonnegative quartic form
$f=f(x,y,z)$ with real coefficients is a sum of three squares of
quadratic forms. His proof was ahead of its time and used advanced
methods from topology and algebraic geometry. In our 2012 paper
\cite{ps} we presented a new approach that used only elementary
techniques. In this note we are adding some further simplifications
to this proof.
\end{abstract}

\maketitle


\section*{Introduction}

This note is an informal complement to our joint paper \cite{ps} from
2012. The purpose of that paper was to present a new approach to
Hilbert's celebrated theorem, according to which every nonnegative
ternary quartic form $f=f(x,y,z)$ over the real numbers can be
written as a sum of three squares of quadratic forms. Other than all
proofs given before, our proof was elementary, in the sense that it
did not use advanced tools from either topology or algebraic
geometry. In principle, the proof should be understandable by a
second year undergraduate.

Although the details of our approach were admittedly complicated in
some parts, the arguments used only basic algebra and calculus. There
remained, however, an annoying point: For one key step, we needed to
verify that a certain invariant (of triples of univariate
polynomials) does not vanish identically. See below for the details.
Easy as this sounds, it presents a problem if one is trying to argue
by paper and pencil only, due to the huge size and complicatedness of
that invariant. Therefore, we used a computer algebra system in
\cite{ps}, to verify the non-vanishing of the invariant for a
suitable example triple.

In Section~1 of this note we present a slight variation of the
approach in \cite{ps}. Combined with the study of a carefully
chosen example triple (Section~2), this does indeed allow to give the
proof, using no other tools beyond paper and pencil.


\section{A slightly simpler invariant}

We start with a very brief summary of the approach in \cite{ps}.
Let $f=f(x,y,z)$ be a ternary quartic form over $\R$ that is
positive semidefinite (\emph{psd}), meaning that $f(a)\ge0$ for
every $a\in\R^3$. The goal is to show that there exist quadratic
forms $q_1,\,q_2,\,q_3$ in $\R[x,y,z]$ such that
\begin{equation}\Label{so3s}%
f\>=\>q_1^2+q_2^2+q_3^2.
\end{equation}
By a limit argument \cite[Lemma 3.6]{ps}, it suffices to prove the
existence of a representation \eqref{so3s} for generic psd $f$,
i.e., for every psd form $f$ that satisfies a condition
$\Psi(f)\ne0$ where $\Psi$ is a suitable fixed nonzero polynomial
in the coefficients of $f$.
Assuming that the psd form $f$ has no nontrivial real zero, we
constructed an explicit psd form $f^{(0)}$ with a nontrivial real
zero and such that, writing $f^{(t)}=tf+(1-t)f^{(0)}$, the form
$f^{(t)}$ is strictly positive for every $t$ with $0<t\le1$. Under a
number of explicit generic assumptions on $f$, we proved that every
representation \eqref{so3s} of $f^{(0)}$ can be extended continuously
to a representation of $f=f^{(1)}$ along the path $f^{(t)}$, thereby
proving Hilbert's theorem.

The hardest step in this approach was to prove the following
statement%
\footnote
  {we dehomogenize as in \cite{ps}, thereby passing from binary forms
  to univariate polynomials}
(\cite[Prop.~8.1]{ps}):

\begin{prop}\Label{prop8.1}%
Consider triples $(f_2,f_3,f_4)$ of univariate polynomials in
$\R[x]$, with $\deg(f_i)\le i$ for $i=2,3,4$, for which there exist
polynomials $\xi\in\R[x]_{\le2}$, $\eta\in\R[x]_{\le3}$ with
\begin{equation}\Label{}%
\eta^2\>=\>(f_2-t\xi)(4f_4-\xi^2)-f_3^2
\end{equation}
for some real number $t\ne0$, and such that the univariate
polynomials
\begin{equation}\Label{8.0ab}%
g_t(\xi)\>=\>t\xi^3-f_2\xi^2-4tf_4\xi+4f_2f_4-f_3^2,\quad
h_t(\xi)\>=\>3t\xi^2-2f_2\xi-4tf_4
\end{equation}
have a common irreducible quadratic factor. Then these triples are
not Zariski dense (in the space of all triples of the given degrees).
\end{prop}

To prove Proposition \ref{prop8.1}, we identified an explicit
polynomial in the coefficients of $f_2,f_3,f_4$, with the property
that it vanishes for all triples as in \ref{prop8.1}. The hard point
then was to show that this polynomial is not identically zero. More
specifically, the following was done in \cite{ps}.

\begin{lab}\Label{recallPhi}%
Fix integers $m,\,n\ge2$ and consider triples $(f,g,h)$ of univariate
polynomials in $\R[x]$ with $\deg(f)\le m$ and
$\deg(g),\,\deg(h)\le n$. In \cite[Sect.~7]{ps} we constructed an
explicit form $\Phi_{m,n}(f,g,h)$ in the coefficients of $f,\,g$ and
$h$, and with $\Z$-coefficients, for which the following holds:
Whenever $f\in\R[x]_{\le m}$ is separable with $\deg(f)\ge m-1$, and
$g,\,h\in\R[x]_{\le n}$ are linearly independent, the vanishing of
$\Phi_{m,n}(f,g,h)$ is equivalent to the existence of a pair
$(s,t)\ne(0,0)$ of complex numbers such that $f$ and $sg+th$ have a
common factor of degree at least two \cite[Cor.~7.3]{ps}. Speaking
informally, the invariant $\Phi$ allows to decide whether a given
polynomial has two roots in common with some member in a given pencil
of polynomials.

In principle, the construction of $\Phi_{m,n}$ in \cite{ps} is
explicit. In practice, however, it is close to impossible to
calculate $\Phi_{m,n}(f,g,h)$ by hand for general triples $(f,g,h)$,
even if $m$ and $n$ are only moderately large. This is due to the
enormous size of the invariant $\Phi_{m,n}$, in terms of its degree
and number of monomials (as a polynomial in the coefficients of
$f,g,h$), as well as in terms of its values.
\end{lab}

\begin{lab}\Label{assmpts}%
To describe how Proposition \ref{prop8.1} was proved in \cite{ps}
using a $\Phi_{m,n}$-invariant, assume now that
$f_2,\,f_3,\,f_4,\,\xi,\,\eta$ are polynomials in $\R[x]$ of
respective degrees at most $2,\,3,\,4,\,2,\,3$, and that $t\ne0$ is
a real number. For $i,j\in\{2,3,4\}$ put
$$g_{ij}\>=\>if_if'_j-jf_jf'_i.$$
The $g_{ij}$ satisfy the relation
\begin{equation}\Label{(8.8)}
2f_2g_{34}-3f_3g_{24}+4f_4g_{23}\>=\>0,
\end{equation}
see \cite{ps} 8.4. As in \cite{ps} 8.6--8.9, we assume that
\begin{equation}\label{nonSi}%
\gcd(f_3,f_4)=1,\ \gcd(g_{23},g_{24})=1,\
\gcd(g_{34},g_{24})=1,\ \Phi_{3,4}(f_3,f_2^2,f_4)\ne0,
\end{equation}
and also that $f_3$ and $f_6=4f_2f_4-f_3^2$ are separable.

Under these assumptions let $p\in\R[x]$ be an irreducible quadratic
polynomial that divides both polynomials $g_t(\xi)$ and $h_t(\xi)$
from \eqref{8.0ab}, for some $t\ne0$. Then the product
$\disc_8(P)\cdot\Phi_{8,11}(P,Q,R)$ was shown to vanish in
\cite{ps} 8.9, for the polynomials
\begin{equation}\Label{defP}%
P\>:=\>g_{24}^2-g_{23}g_{34},
\end{equation}
$$Q\>:=\>4f_3f_4g_{24}+(4f_2f_4-3f_3^2)g_{34},$$
$$R\>:=\>f_2(f_3^2g_{23}-4f_2f_3g_{24}+4f_2^2g_{34})$$
(of degrees $\le8,\,11,\,11$, respectively).

It therefore remained to show the existence of a single triple
$(f_2,f_3,f_4)$ for which $\Phi_{8,11}(P,Q,R)\ne0$. Here we used the
help of a computer algebra system (see \cite{ps} 8.11): For a certain
explicit triple $(f_2,f_3,f_4)$, the value of $\Phi_{8,11}(P,Q,R)$
was computed to be a nonzero integer with 225 decimal digits. In
fact, this was the smallest nonzero value of the invariant that we
could find (for $f_2,f_3,f_4$ with integer coefficients), after a
small random search.
\end{lab}

We now present an argument that leads to a modest simplification of
the $\Phi_{8,11}$-invariant:

\begin{prop}\Label{improvenaja}%
Under the assumptions made in \ref{assmpts}, the existence of a
common factor $p$ as in \ref{assmpts} implies $\Phi_{8,9}(P,A,B)=0$,
for $P$ as in \eqref{defP} and for
\begin{equation}\Label{dfnAB}%
A\:=\>g_{23}\cdot\bigl(g_{23}f_3-2g_{24}f_2\bigr),\quad
B\>:=\>g_{24}g_{34}.
\end{equation}
\end{prop}

Since the (generic) degree of $A$ and $B$ is $9$, and thereby smaller
than $11$ (the generic degree of $Q$ and $R$), the invariant
$\Phi_{8,9}(P,A,B)$ is expected to be slightly more accessible than
$\Phi_{8,11}(P,Q,R)$.
In Section~2 we are going to prove that this new invariant does not
vanish identically, in a way that can entirely be done by hand.

\begin{lab}
Here is the proof of Proposition \ref{improvenaja}. Under the
assumptions made in \ref{assmpts}, let $p\in\R[x]$ be an irreducible
quadratic polynomial that divides both $g_t(\xi)$ and $h_t(\xi)$. As
in \cite{ps} we denote congruences in $\R[x]$ modulo the principal
ideal $(p)$ by $\equiv$, so effectively we are doing calculations in
$\R[x]/(p)\cong\C$. By \cite{ps} 8.8 we have
$$g_{23}g_{34}-g_{24}^2\>\equiv\>0.$$
Conditions \eqref{nonSi} imply that $g_{23}$, $g_{24}$,
$g_{34}\not\equiv0$ (cf.\ \cite{ps} 8.8).
Let $h\in\R[x]$ be a polynomial with
$g_{23}h\equiv g_{24}$. Then $P\equiv0$ implies
\begin{equation}\Label{5}%
\frac{g_{24}}{g_{23}}\>\equiv\>\frac{g_{34}}{g_{24}}\>\equiv\>h,
\quad\frac{g_{34}}{g_{23}}\>\equiv\>h^2.
\end{equation}
Dividing \eqref{(8.8)} by $g_{23}$ we get
\begin{equation}\Label{7}
2h^2f_2-3hf_3+4f_4\>\equiv\>0.
\end{equation}
By \cite{ps} (8.15) we have
$$(2f_2g_{24}-f_3g_{23})^2\>\equiv\>8t^2g_{24}
\bigl(2f_4g_{24}-f_3g_{34}\bigr),$$
and dividing both sides by $g_{23}^2$ gives
\begin{equation}\Label{6}%
(2hf_2-f_3)^2\>\equiv\>8t^2h^2(2f_4-hf_3).
\end{equation}
\end{lab}

\begin{lab}\Label{step2}%
We claim that \eqref{6} does not vanish modulo~$p$. Indeed, assume
to the contrary that $2hf_2\equiv f_3$. Then \eqref{7} implies
$4f_4\equiv3hf_3-2h^2f_2\equiv3hf_3-hf_3=2hf_3$, and hence
\begin{equation}\Label{7a}
2f_4\>\equiv\>hf_3,
\end{equation}
and also
\begin{equation}\Label{7b}
f_6\>\equiv2hf_2f_3-f_3^2
\end{equation}
since $f_6=4f_2f_4-f_3^2$. According to \cite{ps} (8.10) we have
$g_{24}(4f_4-\xi^2)\equiv2f_3g_{34}$, and division by $g_{23}$ gives
$h(4f_4-\xi^2)\equiv2h^2f_3$. Cancelling $h$ (note that
$h\not\equiv0$) we get $4f_4-\xi^2\equiv2hf_3$. By \eqref{7a} this
says $\xi^2\equiv0$, and so $\xi\equiv0$. It follows that
$f_6\equiv f'_6\equiv0$ since $f_2\xi^2+8tf_4\xi\equiv3f_6$
(\cite{ps} (8.4)) and $f'_6\equiv f'_2\xi^2+4tf'_4\xi$ (\cite{ps}
(8.6)). So the discriminant of $f_6$ vanishes, contradicting one of
the assumptions made in \ref{assmpts}.
\end{lab}

\begin{lab}
According to \eqref{7} we have
$$4f_4-2hf_3\>\equiv\>(3hf_3-2h^2f_2)-2hf_3\>\equiv\>h(f_3-2hf_2).$$
Substituting this into the right hand side of \eqref{6}, we get
\begin{equation}\Label{8a}%
(2hf_2-f_3)^2\>\equiv\>4t^2h^2\cdot h(f_3-2hf_2)\>=\>
4t^2h^3(f_3-2hf_2).
\end{equation}
The factor $f_3-2hf_2$ occurs on both sides and may be cancelled
since $2hf_2\not\equiv f_3$ by \ref{step2}. So we conclude
\begin{equation}\Label{8}%
f_3-2hf_2\>\equiv4t^2h^3.
\end{equation}
By \eqref{5}, this says
$$f_3-2\frac{g_{24}}{g_{23}}f_2\>\equiv\>4t^2\frac{g_{24}g_{34}}
{(g_{23})^2}\,,$$
and after multiplying both sides with $(g_{23})^2$, we have shown
$$g_{23}\cdot\bigl(g_{23}f_3-2g_{24}f_2\bigr)-4t^2g_{24}g_{34}
\>\equiv\>0.$$
In other words, we have found a nontrivial linear combination of
\begin{equation}\Label{A}%
A\>=\>g_{23}\bigl(g_{23}f_3-2g_{24}f_2\bigr)
\end{equation}
and
\begin{equation}\Label{B}%
B\>=\>g_{24}g_{34}
\end{equation}
that shares the factor $p$ of degree two with $P$. Since
$\deg(A),\,\deg(B)\le9$, we conclude
$$\Phi_{8,9}(P,A,B)\>=\>0.$$
\end{lab}

We add the remark that $\Phi_{8,9}(P,A,B)$, considered as a
$\Z$-polynomial in the coefficients of $(f_2,f_3,f_4)$, is a form
of degree 476 in 12 variables.


\section{Non-vanishing of the invariant}\Label{part2}%

We start with an easy lemma.
Let $k$ be a field with $\ch(k)=0$ and with algebraic closure
$\ol k$. Let $f,g,h\in k[x]$ be univariate polynomials such that $f$
is separable,
and let $m,\,n$ be integers with $\deg(f)=m$ and
$\deg(g),\,\deg(h)\le n$.

\begin{lem}\Label{3novlem}%
Assume that $\Phi_{m,n}(f,g,h)=0$ and that $g,h$ are relatively
prime. Then there exist roots $\alpha\ne\beta$ of $f$ in $\ol k$ and
elements $a,\,b\in k(\alpha)\cap k(\beta)$, not both zero, such that
$\alpha,\beta$ are roots of $p=ag+bh$. In particular, $p$ has
coefficients in the subfield $k(\alpha)\cap k(\beta)$ of~$\ol k$.
\end{lem}

\begin{proof}
$\Phi_{m,n}(f,g,h)=0$ implies that there are $s,t\in\ol k$, not both
zero, such that $\alpha,\beta$ are roots of $sg-th$
(\cite[Cor.~7.3]{ps}, see \ref{recallPhi}).
We may assume that $t=1$,
and so $sg(\alpha)-h(\alpha)=sg(\beta)-h(\beta)=0$. Since $g,h$ are
relatively prime we have $g(\alpha)g(\beta)\ne0$.
Therefore $s=\frac{h(\alpha)}{g(\alpha)}=\frac{h(\beta)}{g(\beta)}$
lies in $k(\alpha)\cap k(\beta)$, and we may take $a=s$, $b=-1$.
\end{proof}

\begin{lab}\Label{calcrems}%
After a longer search for suitable example triples $(f_2,f_3,f_4$)
(which was carried out with the help of a computer algebra system, in
this case \texttt{MAPLE}), we came up with the following one. The
advantage of this triple is that the polynomial $P$ derived from
$(f_2,f_3,f_4$) (of degree~8, see \eqref{defP}) splits over $\Q$ into
factors of degree at most three, with only one cubic factor.

The example is as follows. Take
$$f_2=2x^2-1,\quad
f_3=2x^3-x^2-2x+1,\quad
f_4=x^4+x^3-2x^2+x+1.$$
Then $P=g_{24}^2-g_{23}g_{34}=56x^8-52x^7+180x^6-40x^5+40x^4+284x^3
+84x^2+128x-40$ factors as
\begin{equation}\Label{factorsP}%
P\>=\>4(x+1)(2x^2+x+1)(x^2-2x+2)(7x^3-3x^2+21x-5)
\end{equation}
and has only two real zeros. Denote the factors of $P$ by $\ell=x+1$
and by
\begin{equation*}
q_1=2x^2+x+1,\quad q_2=x^2-2x+2,\quad g=7x^3-3x^2+21x-5.
\end{equation*}
Then $q_1,\,q_2,\,g$ are irreducible over $\Q$, and $g$ has only one
real zero. The polynomials
$$A=96x^9-16x^8-160x^7+704x^6-208x^5-512x^4+208x^3+208x^2-128x$$
and
$$B=40x^9-108x^8+316x^7-198x^6+316x^5+122x^4+180x^3-178x^2-84x-22$$
derived from $f_2,f_3,f_4$ (see \eqref{A}, \eqref{B}) reduce
modulo $q_1$, $q_2$ and $g$ as follows:
\begin{equation}\Label{modq1}%
A\>\equiv\>-\frac14(2389\,x+271)\ \text{ (mod }q_1),\quad
B\>\equiv\>\frac1{16}(741\,x+1471)\ \text{ (mod }q_1),
\end{equation}
\begin{equation}\Label{modq2}%
A\>\equiv\>-1280\,x+3616\ \text{ (mod }q_2),\quad
B\>\equiv\>-1648\,x+870\ \text{ (mod }q_2)
\end{equation}
and
$$A\>\equiv\>\frac1{7^7}\bigl(3869324320\,x^2+9251095616\,x-2476940000\bigr)\
\text{ (mod }g),$$
\begin{equation}\Label{modg}%
B\>\equiv\>\frac1{7^7}\bigr(818130160\,x^2-1744372672\,x+333091504\bigr)\
\text{ (mod }g).
\end{equation}
Using polynomial division with remainder, congruences \eqref{modq1},
\eqref{modq2} and \eqref{modg} can be found using only paper and
pencil (plus, for sure, a certain amount of diligence!). Since $A$
and $B$ are visibly linearly independent modulo $q_1$ (congruences
\eqref{modq1}), there cannot be any non-trivial linear combination of
$A$ and~$B$ that is divisible by $q_1$. A similar argument works when
$q_1$ is replaced with $q_2$ or with~$g$.
\end{lab}

\begin{lab}
We claim that the previous discussion already implies
$\Phi_{8,9}(P,A,B)\ne0$.
To see this, note first that $A$ and $B$ are relatively prime (again,
this is easily verified by hand using the Euclidean algorithm).
So Lemma \ref{3novlem} can be applied. By \ref{recallPhi} it suffices
to show that $P$ has at most one root in common with an arbitrary
non-zero $\C$-linear combination of $A$ and~$B$. Assume to the
contrary that $P$ has two roots $\alpha\ne\beta$ in common with
$p=aA+bB$, where $(a,b)\ne(0,0)$ is a pair of complex numbers. If
$\alpha,\,\beta$ were the
roots of $q_1$ then $q_1$ would divide $p$, contradicting
\eqref{modq1}. In a similar way one sees that $\alpha,\,\beta$ cannot
be the roots of $q_2$. Looking at the factorization \eqref{factorsP}
of $P$, one sees that in all other possible cases we have
$\Q(\alpha)\cap\Q(\beta)=\Q$. (Observe for this that the splitting
field of $g$ has degree~$6$ over $\Q$, so any two different roots of
$g$ generate different cubic extensions of $\Q$.) Therefore
$p$ has rational coefficients by Lemma \ref{3novlem}. This implies
that $p$ is divisible by at least one of $q_1,\,q_2$ or $g$. Again,
this is impossible by the discussion in \ref{calcrems}.
\end{lab}

To summarize, by analyzing the above triple $(f_2,f_3,f_4)$ with
paper-and-pencil arguments only, we concluded that the invariant
$\Phi_{8,9}(P,A,B)$ is non-zero (thereby implying Proposition
\ref{prop8.1} by Proposition \ref{improvenaja}). The exact value of
this invariant, calculated with the help of a computer algebra
system, happens to be the integer

\vspace{2mm}
{\Small
\begin{center}
\begin{tabular}{r}
170180100414489407673826285238621248588184132495664769101548147694597
\textbackslash \\
641645055149834797961367009741001058378563516737825717521245942079
\textbackslash \\
363665365894932768287485782991982952060121491854462396585226867885
\textbackslash \\
300239619184714256923401363159130009392223954249957235784417280000
\phantom{\textbackslash}
\end{tabular}
\end{center}
}%
\vspace{2mm}
\noindent
with 267 decimal digits and with prime factorization

{\Small
$$2^{368}\cdot3^{68}\cdot5^4\cdot7^4\cdot11^2\cdot13^{30}\cdot17^3
\cdot29^2\cdot43^{28}\cdot53\cdot137\cdot389\cdot577\cdot1381\cdot
1657\cdot11173\cdot18757\cdot121349.$$%
}%


\end{document}